\documentclass[11pt]{amsart}
\usepackage{graphicx}
\usepackage[active]{srcltx}
 \makeatletter
\renewcommand*\subjclass[2][2000]{%
  \def\@subjclass{#2}%
  \@ifundefined{subjclassname@#1}{%
    \ClassWarning{\@classname}{Unknown edition (#1) of Mathematics
      Subject Classification; using '1991'.}%
  }{%
    \@xp\let\@xp\subjclassname\csname subjclassname@#1\endcsname
  }%
}
 \makeatother
\usepackage{enumerate,url,amssymb,  mathrsfs}

\newtheorem{theorem}{Theorem}[section]
\newtheorem{lemma}[theorem]{Lemma}
\newtheorem*{lemma*}{Lemma}
\newtheorem{proposition}[theorem]{Proposition}
\newtheorem{corollary}[theorem]{Corollary}

\theoremstyle{definition}

\theoremstyle{remark}
\newtheorem{remark}[theorem]{Remark}

\numberwithin{equation}{section}

\def\XXint#1#2#3{{\setbox0=\hbox{$#1{#2#3}{\int}$}
\vcenter{\hbox{$#2#3$}}\kern-.5\wd0}}

\def\le{\leqslant}
\def\ge{\geqslant}
\begin{document}
\title[Muckenhoupt weights and Lindel\"of theorem]{Muckenhoupt weights and Lindel\"of theorem for harmonic mappings}
\author{David Kalaj}
\address{University of Montenegro, Faculty of Natural Sciences and
Mathematics, Cetinjski put b.b. 81000 Podgorica, Montenegro}
\email{davidkalaj@gmail.com}


\begin{abstract}
We extend the result of Lavrentiev which asserts that the harmonic measure and the arc-length measure are $A_\infty$ equivalent in a chord-arc Jordan domain. By using this result we extend the classical result of Lindel\"of  to the class of quasiconformal (q.c.)  harmonic mappings by proving the following assertion. Assume that $f$ is a quasiconformal harmonic  mapping of the unit disk $\mathbf{U}$ onto a Jordan domain. Then the function $A(z)=\arg(\partial_\varphi(f(z))/z)$ where $z=re^{i\varphi}$, is well-defined and smooth in $\mathbf{U}^*=\{z: 0<|z|<1\}$ and  has a continuous extension to the boundary of the unit disk if and only if the image domain has $C^1$ boundary.
\end{abstract}

\maketitle
\tableofcontents
\section{Introduction and statement of the main results}
\subsection{Quasiconformal mappings}  By definition,  $K$-quasiconformal mappings (or shortly q.c. mappings) are orientation preserving homeomorphisms $f:D \to \Omega$ between domains $D,\Omega \subset \mathbf{C}$, contained  in the Sobolev class $W^{1,2}_{loc}(D)$, for which  the differential matrix   and its determinant are coupled in the distortion inequality,
\begin{equation}\label{distortion}
 |D\!f(z)|^2  \leq K\, \det D\!f(z)\;,\quad \textrm{where}\;\;\;|D\!f(z)|  = \max_{|\xi| =1} \; |D\!f(z) \xi|,
\end{equation}
for some  $K \geq1$. Here $\det D\!f(z)$ is the determinant of the formal derivative $D\!f(z)$.  Note that the
condition (\ref{distortion}) can be written in complex notation as \begin{equation}\label{same}{(|f_z|+|f_{\bar z}|)^2}\le K({|f_z|^2-|f_{\bar z}^2|})\quad \text{a.e. on $D$}\end{equation} or what is the same,  $$|f_{\bar z}|\le
k|f_z|\quad \text{a.e. on $D$ where $k=\frac{K-1}{K+1}$ i.e.,
$K=\frac{1+k}{1-k}$ }.$$

\subsection{Harmonic mappings}
A mapping  $f$ is called \emph{harmonic}
in a region $D$ if it has the form  $f=u+iv$ where $u$ and $v$ are
real-valued harmonic functions in $D$. If $D$ is simply-connected,
then there are two analytic functions $a$ and $b$ defined on $D$
such that $f$ has the representation $$f=a+\overline b.$$

If $f$ is a harmonic univalent function, then by Lewy's theorem
(see \cite{l}), $f$ has a non-vanishing Jacobian and consequently,
according to the inverse mapping theorem, $f$ is a diffeomorphism.

Let $$P(r,x-\varphi)=\frac{1-r^2}{2\pi
(1-2r\cos(x-\varphi)+r^2)}$$ denote the Poisson kernel. If $f^*\in L^1(\mathbf{T})$, where $\mathbf{T}$ is the unit circle, then we define the Poisson integral $\mathcal{P}[f^*]$ of $f^*$ by formula \begin{equation}\label{e:POISSON}
\mathcal{P}[f^*](z)=\int_0^{2\pi}P(r,x-\varphi)f^*(e^{ix})dx, \ \ |z|<1.
\end{equation} The function $f(z)=\mathcal{P}[f^*](z)$ is a harmonic mapping in the unit disk $\mathbf{U}=\{z:|z|<1\}$, which belongs to the harmonic Hardy space $h^1(\mathbf{U})$. The mapping $f$ is bounded in $\mathbf{U}=\{z:|z|<1\}$ if and only if $f^*\in L^\infty(\mathbf{T}).$ We will simultaneously use the notation $f$ for $\mathcal{P}[f^*]$ and its boundary function $f^*$ in order to have a better exposition.

 It is well-known that the Poisson integral  $\mathcal{P}[f^*]$ extends by continuity to $f^*$ on $\overline{\mathbf{U}}$, provided that $f^*$ is continuous. For this fact and standard properties of harmonic Hardy space $h^1(\mathbf{U})$ we refer to \cite[Chapter~6]{abr} and \cite{Pd}. With the additional assumption that $f^*$ is an orientation-preserving homeomorphism of this circle onto a convex Jordan curve $\gamma$, $\mathcal{P}[f^*]$ is an orientation preserving diffeomorphism of the open unit disk. This is actually the celebrated theorem of Choquet-Rado-Kneser (\cite{duren}).
This theorem is not true for non-convex domains, but hold true under some additional assumptions.
It has been extended in various
directions (see for example \cite{jj, studia, ale, dur}).

The class of all conformal mapping of the unit disk onto itself coincides with the set of all M\"obius transformations and thus the class is rigid in some sense. The classes of quasiconformal harmonic mapping with respect to the Euclidean metric proved to be a significant topic of geometric function theory recently. The pioneering work in this subject have been done by Martio in \cite{martio}. We refer to the papers \cite{studia,mathz,pisa, trans,MMM, MP,pk} for some recent progress in this class. In these papers has been in particular treated the Lipschitz and bi-Lipschitz character of this class assuming that the domain and the image domain have $C^{k,\alpha}$  boundaries with $k\ge 1$, $\alpha>0$. In papers \cite{kojic,matvuo}  some regularity results of Quasiconformal harmonic mappings with non-necessary smooth domains have been treated.

For the topic of quasiconformal mappings harmonic w.r.t. the hyperbolic metric and connection with Teichm\"uller spaces we refer to the paper \cite{markovic} and the references therein.

\subsection{The Lindel\"of Theorem}
An analytic characterization of the smoothness of a Jordan domain  is given by the classical
Lindel\"of \cite{lin,martio} theorem:
\begin{proposition} Let $f$ map $\mathbf{U}$ conformally onto the inner domain of a Jordan
curve $\gamma$. Then $\gamma$ is smooth if and only if $\arg\, f'(z)$ has a continuous
extension to ${\mathbf{T}}$, which we denote by $f(e^{i\varphi})$. If $\gamma$ is smooth, then
\begin{equation}\label{lineq}\arg\, f'(e^{i\varphi}) = \beta(\varphi)-\varphi-\frac{\pi}{2}\end{equation} where $\beta(\varphi)$ stands for the tangent angle of the curve $\gamma$ at the point $f(e^{i\varphi})$.
\end{proposition}

\subsection{Lavrentiev theorem}
Assume that $\gamma$ is a rectifiable Jordan curve and denote by $\sigma$ its arc length. The shorter arc between $z$ and $w$ in $\gamma$ will be denoted by $(z,w)$. We say that $\gamma$ is chord-arc curve or Lavrentiev curve if there is a constant $M>1$ such that for $z,w\in\gamma$ we have $\sigma(z,w)\le M|z-w|$.

 To formulate Jerison-Kenig version of Lavrentiev theorem we say that the measures $\omega$ and $\sigma$ on a Jordan chord-arc curve $\gamma$ are $A_\infty$ equivalent if for any $\varepsilon>0$ there is $\delta>0$ such that for every arc $I\subset \gamma$ and every Borel set $E\subset I$, $\omega(E)/\omega(I)<\delta$ implies $\sigma(E)<\sigma(I)<\varepsilon$ (\cite{kenig2}). Coifmann and Fefferman in \cite{coif} proved that $A_\infty$ is an equivalence relation (See the subsection below for more details in the topic). Assume that $\Phi$ is conformal mapping of the unit disk onto a Jordan domain $\Omega$ such that $\Phi(0)=a\in \Omega$. Then we define the harmonic measure in $\partial\Omega$  w.r.t $a$ as follows $\omega(E)=|f^{-1}(E)|$, for a Borel set $E\subset \partial\Omega$, and $|F|$ is the arc-length in the unit circle $\mathbf{T}$.
\begin{proposition}\cite{kenig2}
If $\Omega$ is chord-arc domain, then the harmonic measure and arc-length measure are $A_\infty$ equivalent in $\partial \Omega$.
\end{proposition}

\subsection{Muckenhoupt weights}
 Consider the measurable function $\psi$ on $\mathbf{R}^n$ and its associated maximal function $M(\psi)$ defined as
$$ M(\psi)(x) = \sup_{r>0} \frac{1}{r^n} \int_{B_r} |\psi|,$$
where $B_r$ is a ball in $\mathbf{R}^n$ with radius $r$ and centre $x$. One of important problems of harmonic analysis is to characterize the functions $\omega \colon \mathbf{R}^n \to [0,\infty)$ for which we have a bound

$$\int |M(\psi)(x)|^p \, \omega(x) dx \leq C \int |\psi|^p \, \omega(x)\, dx,$$
where $C$ depends only on $p \in [1,\infty)$ and $\omega$.

The class of "Muckenhoupt weights" $A_p$, $p>1$ consists of those weights $\omega$ for which the Hardy--Littlewood maximal operator is bounded on $L^p(d\omega)$.  For a fixed $1 < p < \infty$, we say that a weight $\omega \colon \mathbf{R}^n \to [0,\infty)$ belongs to the class $\mathbf{A}_p$  if $\omega$ is locally integrable and there is a constant $c$ such that, for all balls $B$ in $\mathbf{R}^n$, we have

\begin{equation}\label{pq}\left(\frac{1}{|B|} \int_B \omega(x) \, dx \right)\left(  \frac{1}{|B|} \int_B \omega(x)^\frac{-q}{p} \, dx \right)^\frac{p}{q} \leq c < \infty,\end{equation}
where $1/p + 1/q = 1$ and $|B|$ is the Lebesgue measure of of $B$.

The  fundamental result in the study of Muckenhoupt weights \cite{stein, muc}, states that $\mathbf{A}_p=A_p$ for all $p>1$.

 Further, the weight $\omega$ satisfies the Feffereman-Coifman  $A_\infty$ condition in the unit circle $\mathbf{T}$ if one of the following two equivalent conditions hold:

(i) For all arcs $I\subset\mathbf{T}$ there holds the inequality $$\frac{1}{|I|}\int_I \omega(t)dt\le M_1\exp\left(\frac{1}{|I|}\int_I\log w(t) dt\right).$$

(ii) For some $0<\varepsilon<1$ there exists $\delta<1$ such that for all arcs $I\subset{T}$ and measurable sets $E\subset I$ we have
\begin{equation}\label{ii}\frac{\int_E \omega(t) dt}{\int_I \omega(t) dt}<\delta\Rightarrow \frac{|E|}{|I|}<\delta.\end{equation}

It is well known that if $\omega $ satisfies $A_\infty$ condition then it satisfies $A_p$ condition for some $p>1$ (\cite{coif}). Thus $$A_\infty=\bigcup_{p<\infty} A_p.$$ Further, $\omega $ satisfies $A_\infty$ condition if and only if it (\cite[p.~170]{garnet}) satisfies
Gehring  $B_q$ condition for some $q>1$: for all arcs $I\subset\mathbf{T}$ there holds the inequality $$\left(\frac{1}{|I|}\int_I \omega(t)^qdt\right)^{1/q}\le \frac{M_2}{|I|}\int_I w(t) dt.$$

\subsection{New results}
The aim of this paper is to prove the following extension of Lindel\"of theorem

\begin{theorem}\label{lindel}
If $f(z)=\mathcal{P}[f^*](z)$ is a quasiconformal harmonic mapping of the unit disk onto a Jordan domain bounded by a curve  $\gamma$, then the function $$U(z):=\arg\left(\frac{1}{z}\frac{\partial}{\partial \varphi}f(z)\right)$$ is well defined and smooth in $\mathbf{U}^*:=\mathbf{U}\setminus\{0\}$ and has a continuous extension to $\mathbf{T}$ if and only if $\gamma\in C^1$. Furthermore, there holds $$U(e^{i\varphi})=\beta(\varphi)-\varphi,$$  where  $\beta(\varphi)$ is the tangent angle of $\gamma$ at $f^*(e^{i\varphi})$.
\end{theorem}
\begin{remark}
 In order to deduce the classical Lindel\"of theorem from Theorem~\ref{lindel}, observe that, if $f=\mathcal{P}[F]$ is conformal, which is then certainly quasiconformal and harmonic, for $z=re^{i\varphi}$ we have $$\partial_\varphi f(z)=i zf'(z).$$ We infer that $$\arg(\partial_\varphi f(z))=\frac{\pi}{2}+\varphi+\mathrm{arg}(f'(z)).$$ Thus on the unit circle we have $$\mathrm{arg}(f'(z))=\beta(\varphi)-\frac{\pi}{2}-\varphi,$$ which coincides with \eqref{lineq}.

\end{remark}
To prove Theorem~\ref{lindel}, we prove the following extension of Jerison-Kenig  version (\cite{kenig1}) of Lavrentiev theorem (\cite{lavrentiev}).
\begin{theorem}\label{onemaine}
Assume that $f$ is a $K-$ quasiconformal harmonic mapping of the unit disk onto a domain $\Omega$ bounded by a chord-arc Jordan curve. Then all the following four (equivalent) statements hold for the weight  $w(t)=|\partial_t f(e^{it})|$.
  \begin{enumerate}
   \item  $w(t)$ satisfies $A_\infty$ Coifman-Fefferman condition;
   \item $w(t)$ is a Muckenhoupt weight;
   \item $w(t)$ satisfies Gehring condition;
   \item $w(t)dt$ is $A_\infty$ equivalent to the arc-length measure $dt$.
 \end{enumerate}
 \end{theorem}
The $\textbf{BMO}$ space of functions on the circle $\mathbf{T}$ is defined as follows.
Let $\psi\in L^1(\mathbf{T} )$. We say $\psi\in \textbf{BMO}(\textbf{T} )$ if
$$\sup_I\frac{1}{|I |}\int_{I}|\psi(t)-\psi_I| dt=\|\psi\|_\ast<\infty $$
where $I$ denotes any arc on $\mathbf{T}$, $|I|$ is the length of $I$, and $$\psi_I =\frac{1}{|I|}\int_I \psi(x) {dt}$$
is the average of $\psi$ over $I$. The space of analytic functions $f\in H^1(\mathbf{U})$, with boundary function $f\in \mathbf{BMO}(\mathbf{T})$ is denoted by $\mathbf{BMOA}$. Then by using Theorem~\ref{onemaine} and  \cite[Lemma~5]{hun} and \cite[p.~171]{boundary} we immediately  have the following result
\begin{theorem}\label{poqa}
Assume that $f$ is a $K-$ quasiconformal harmonic mapping of the unit disk onto a domain $\Omega$ bounded by a chord-arc Jordan curve.
Then
\begin{itemize}
  \item there is a positive number $\kappa>0$ such that $$\int_{\mathbf{T}}|\partial_t f(e^{it})|^{-\kappa} dt<\infty,$$
  \item there is a positive number $\lambda>1$ such that $$\int_{\mathbf{T}}|\partial_t f(e^{it})|^{\lambda} dt<\infty,$$
  \item  $\log |\partial_t f(e^{it})|\in \mathbf{BMO}(\mathbf{T})$, and
  \item $\log f_z\in \mathbf{BMOA}$.
\end{itemize}
\end{theorem}
\begin{remark}
It follows from the previous corollary that $\arg f_z(e^{it})\in \mathbf{BMO}(\mathbf{T})$. Thus
$$\arg f_z(z)=\int_{\mathbf{T}}\frac{1-|z|^2}{|z-e^{it}|^2} \arg f_z(e^{it})\frac{dt}{2\pi}.$$
We want to note that the problem of defining  $\arg f_z$ for general q.c. mappings is a subtle problem. For an approach to the solution of this problem we refer to a recent paper \cite{aips}, where among the other results it has been proved the following sharp result: $e^{b\arg f_z}\in L^1_{loc}(D,\mathbf{C})$ for $0\le b< \frac{4K}{K^2-1}$.
\end{remark}
\begin{proof}[Proof of Theorem~\ref{poqa}]
First two items are immediate consequences of Theorem~\ref{onemaine} and the following proposition.
 \begin{proposition}\cite{MP, kmm}\label{propi}
If $f=a+\overline{b}$ is a q.c. mapping of the unit disk onto a Jordan domain with rectifiable boundary, then $g$ and $h$ have absolutely continuous extension to $\mathbf{T}$.
\end{proposition}
Observe further that a real function $g\in \mathbf{BMO}(\mathbf{T})$ if and only if $\tilde g \in \mathbf{BMO}(\mathbf{T})$, where $\tilde g$ is the Hilbert transform of $g$. Further, both of the previous conditions are equivalent to the fact that $e^{qg(t)}\in A_\infty$ for some $q>0$. Since $|\partial_t f(e^{it})|\in A_\infty$  it follows that  $\log[|\partial_t f(e^{it})|]\in \mathbf{BMO}(\mathbf{T}).$ Since $$\frac{1}{K}|\partial_t f(e^{it})|\le |f_z(e^{it})|\le {K}|\partial_t f(e^{it})|,$$ we infer that $|f_z(e^{it})|\in A_\infty$ and thus
$$\log |f_z(e^{it})|\in \mathbf{BMO}(\mathbf{T}).$$ It follows that $\arg f_z\in \mathbf{BMO}(\mathbf{T})$ and hence, $\log f_z\in \mathbf{BMOA}$.
\end{proof}
 In the second section we will prove some auxiliary results for quasiconformal harmonic mappings that are analogous to related results for harmonic measure and conformal mappings. The main results are proved in the third section. In the proof of Theorem~\ref{onemaine} we follow some ideas from the fundamental paper by Jerrison and Kenig \cite{kenig2}. We will only prove the first statement of Theorem~\ref{onemaine} (see \eqref{ii}), and this suffices. The core of the proof is Lemma~\ref{apo}, where the arc $I$ is replaced by $\mathbf{T}$. Then we will reduce the whole proof of the main result to Lemma~\ref{apo} by using the Distortion theorem for global quasiconformal mappings, and some known auxiliary results for the class of quasiconformal harmonic mappings. One of difficulties that appear in this new setting is the fact that the logarithm of a harmonic mapping is not a harmonic mapping, in general. The proof of Lindel\"of theorem  for q.c. harmonic mappings (Theorem~\ref{lindel}) depends on Theorem~\ref{onemaine}, but it involves a subtle analysis of Poisson integral formula. This proof is completely different from the known  proofs of Lindel\"of theorem for conformal mappings \cite{martio,lin}.

\section{Auxiliary results}
\begin{lemma}\label{poki}
Let $f(z)=a+\overline{b}$, $b(0)=0$, be a $K-$quasiconformal mapping of the unit disk onto the domain $\Omega$. Then $$|f_z(0)|\ge \frac{3\sqrt{3}(1+K)}{2\sqrt{2}\pi\sqrt{1+K^2}}\mathrm{dist}(f(0),\partial \Omega).$$
\end{lemma}
\begin{proof}
Let $D=D(f(0),\rho)$, where $\rho=\mathrm{dist}(f(0),\partial \Omega)$ and define $U_1=f^{-1}(D)$. Let $\Phi$ be a conformal mapping of the unit disk onto $U_1$ such that $\Phi(0)=0$. Then $g=\frac{1}{\rho}(f\circ \Phi-f(0))$ is a $K-$quasiconformal harmonic mapping of the unit disk onto itself with $g(0)=0$. By Heinz inequality proved by Hall (\cite{hal,duren}), we have $$|g_z(0)|^2+|g_{\bar z}(0)|^2=\frac{|\Phi'(0)|^2}{\rho^2}(|f_z(0)|^2+|f_{\bar z}(0)|^2\ge \frac{27}{4\pi^2}.$$  By Schwarz inequality we have  $|\Phi'(0)|\le 1$.  Since $f$ is $k=\frac{K-1}{K+1}-$q.c. it follows that $$|f_z(0)|^2\ge \frac{27\rho^2}{(1+k^2)4\pi^2}$$ and this concludes the proof of the lemma.
\end{proof}

\begin{lemma}\label{apo}
Let $f=a+\overline{b}$ be a $K-$quasiconformal mapping of the unit disk $\mathbf{U}$  onto a Jordan domain $\Omega$  bounded by a rectifiable boundary and containing the disk $D(f(0), C_K)$, where \begin{equation}\label{ck}C_K=\frac{2\pi}{3\sqrt{6}}\sqrt{1+K^2}.\end{equation}
Let $|F|$ be the arc-length of $F\subset \mathbf{T}$ and $\sigma(E)$ be the arc-length of $E\subset \partial\Omega$ and define quasi-harmonic measure of $E$ by $\omega_f(E)=|f^{-1}(E)|$.

Then  \begin{enumerate}
       \item for every $\varepsilon>0$ there exists $\delta=\delta(\epsilon)$ such that if $E\subset \partial\Omega$ is measurable,
       then $\sigma(E)<\delta$ implies $\omega_f(E)\le \varepsilon |\partial \Omega|,$ or what is the same,
       \item for every $\varepsilon>0$ there exists $\delta=\delta(\epsilon)$, such that if $F\subset
       \partial D$ is measurable, then  $\sigma(f(F))<\delta$ implies $|F|\le \varepsilon \sigma(\partial \Omega).$
     \end{enumerate}
\end{lemma}
\begin{proof}
Since $\gamma=\partial\Omega$ is rectifiable, then by Proposition~\ref{propi}, $f$ is absolutely continuous on the boundary. We then have that $$\int_0^{2\pi} |\partial_t f(e^{it})| dt=|\partial\Omega|.$$ For $E=f(F)\subset\partial\Omega$, where $F$ is a measurable subset of $\mathbf{T}$, we have  $$\sigma(E)=\int_{f^{-1}(E)}|\partial_t f(e^{it})| dt.$$ Further, since $x\le \log^+ x:=\max\{0,\log x\}$ for $x\ge 0$, it follows that
$$\int_0^{2\pi} \log^+|\partial_t f(e^{it})| dt\le |\partial\Omega|.$$
Furthermore, by main value inequality for harmonic functions and Lemma~\ref{poki} we have \[\begin{split}\int_0^{2\pi} \log|\partial_t f(e^{it})| \frac{dt}{2\pi}&=\int_0^{2\pi} \log|e^{it}a'(e^{it})+\overline{b'(e^{it})e^{it}}| \frac{dt}{2\pi}\\ &\ge \int_0^{2\pi} \log[(1-k)|a'(e^{it})|\frac{dt}]{2\pi}\\&\ge \log[(1-k)|a'(0)|]\ge \log 1=0. \end{split}\]
Now if $F\subset \mathbf{T}$ we have $$\int_F \log|\partial_t f(e^{it})| \frac{dt}{2\pi}\ge -\int_{\mathbf{T}\setminus F} \log|\partial_t f(e^{it})| \frac{dt}{2\pi}$$ and \[\begin{split}-\int_{\mathbf{T}\setminus F} \log|\partial_t f(e^{it})| \frac{dt}{2\pi}&\ge -\int_{\mathbf{T}\setminus F} \log^+|\partial_t f(e^{it})| \frac{dt}{2\pi}\\&\ge -\int_{\mathbf{T}} \log^+|\partial_t f(e^{it})| \frac{dt}{2\pi}\ge -\frac{|\partial\Omega|}{2\pi}.\end{split}\]
Furthermore, $$\int_F \log|\partial_t f(e^{it})| {dt}\ge -{|\partial\Omega|}.$$  By Jensen's inequality we obtain
$$e^{\frac{-{|\partial\Omega|}}{|F|}}\le \int_F |\partial_t f(e^{it})| \frac{{dt}}{|F|}.$$ So $$|F|e^{\frac{-{|\partial\Omega|}}{|F|}}\le \sigma(f(F))$$  i.e., $$|F|\le e^{\frac{{|\partial\Omega|}}{|F|}} \sigma(f(F)).$$ Hence, for $E=f(F)$  we have
$$\frac{1}{\sigma(E)}\le \frac{ e^{\frac{{|\partial\Omega|}}{|f^{-1}(E)|}}}{|f^{-1}(E)|}$$ and so
$$|\log {\sigma(E)}|\le \frac{{|\partial\Omega|}}{|f^{-1}(E)|}-\log |f^{-1}(E)|.$$ Since $x\log 1/x \le 1\le |\partial\Omega|$ we obtain finally that
$$\omega_f(E)\le \frac{2|\partial\Omega|}{|\log {\sigma(E)}|}.$$
\end{proof}

\begin{corollary}\label{apo1}
Let $f=a+\overline{b}$ be a $K-$ quasiconformal mapping of  the unit square $\mathbf{K}=[0,1]\times[0,1]$ onto a Jordan domain $\Omega$  bounded by a rectifiable boundary and containing the disk $D(f(0), C_K)$. Let $|F|$ be the arc-length of $F\subset\partial \mathbf{K}$ and $\sigma(E)$ be the arc-length of $E\subset \partial\Omega$.

Then for every $\varepsilon>0$ there exists $\delta=\delta(\epsilon)$ such that for every $F\subset \partial \mathbf{K}$  we have $$\sigma(f(F))<\delta \Rightarrow|F|\le \varepsilon \sigma(\partial \Omega).$$
\end{corollary}
\begin{proof}
In order to deal with the square $\mathbf{K}$, take a conformal mapping $\varphi$ of the unit disk $\mathbf{U}$ onto $\mathbf{K}$ which maps the origin to the center of $\mathbf{K}$. Then it is well-known  that the harmonic measure of the square $\mathbf{K}$ and the arc-length are mutually absolutely continuous.

More precisely $$\varphi(z)=\frac{1+i}{2}+\frac{(2-2 i) \sqrt{2 \pi }}{\Gamma\left[\frac{1}{4}\right]^2}\int_0^z\frac{1}{\sqrt{1-z^4}}dz$$ is a conformal mapping of the unit disk onto the square mapping the origin to the center of the square. Let $\phi(z)=\varphi^{-1}$. Then for every $\varepsilon>0$ there is $\varepsilon_1>0$ such that if $F\subset \partial \mathbf{K}$, then $$|\phi(F)|\le \varepsilon_1(\sigma(\partial\Omega))\Rightarrow |F|\le \varepsilon\sigma(\partial\Omega).$$

Then the mapping $f_0=f\circ \varphi$ is a $k-$ quasiconformal harmonic mapping of the unit disk onto $\Omega$. By using Lemma~\ref{apo}, for $\varepsilon_1>0$, there is $\delta=\delta(\varepsilon_1)$, such that $\phi(F)\subset \mathbf{T}$ and $|f_0(\phi(F))|<\delta$ implies $|\phi(F)|< \varepsilon_1\sigma(\partial\Omega)$. Thus for $\varepsilon>0$ there is $\delta=\delta(\varepsilon)$ such that for   $F\subset \partial\mathbf{K}$ we have $$\sigma(f(F))\le \delta\Rightarrow |F|< \varepsilon\sigma(\partial\Omega).$$  \end{proof}

\begin{lemma}\label{leqa}
If $f$ is a $K-$q.c. harmonic mapping of the unit disk onto the domain $D$, then $$|Df(z)|\cong \frac{\mathrm{dist}(f(z),\partial D)}{1-|z|}.$$ Here and in the sequel, $L\cong R$ means $C^{-1} L\le R\le C R$, where the constant $C$ depends only on $K$.
\end{lemma}
\begin{proof}
By the Cauchy inequality for harmonic functions for $|z|<1$ we obtain that $$|Df(z)|\le C \frac{\mathrm{dist}(f(z),\partial D)}{1-|z|}.$$ Let $|z|<1$ and $\rho=\mathrm{dist}(f(z),\partial D)$. Let $U_1=f^{-1}( D(f(z),\rho))$ and let $\varphi$ be a conformal mapping of the unit disk onto $U_1$ with $\varphi(0)=z$. Then
$F(w)=\frac{1}{{\rho}}({f(\varphi(w))}-f(z))$ is a q.c. harmonic mapping of the unit disk onto itself satisfying $F(0)=0$. Thus by a known result (see e.g., \cite{MMM}) $$|D(F)(w)|\cong\frac{1-|F(w)|}{1-|w|}.$$ Since $|DF(w)|=|\varphi'(w)| |Df(\varphi(w))|$ and $$\frac{1}{|\varphi'(w)|}\ge \frac{1-|w|^2}{1-|\varphi(w)|^2},$$ it follows that $$|Df(z)|\ge \rho |DF(0)|\frac{1-|0|^2}{1-|\varphi(0)|^2}\ge C \rho \frac{{1-|F(0)|}}{1-|z|^2}=\frac{C\mathrm{dist}(f(z),\partial D)}{1-|z|^2}.$$
\end{proof}

\begin{proposition}[Distortion Theorem] \cite[p.~63]{vais} and \cite[p.~383]{ge}\label{poli}.
Let $f$ be a $K$ q.c. mapping of the complex plane $\mathbf{C}$ onto itself. Then there is a constant $C$ depending only on $K$ such that for any $r>0$ and $z\in \mathbf{C}$ there is $r'>0$ such that \begin{equation}\label{disto}
D(f(z), r')\subset f(D(z,r))\subset D(f(z),C r'),
\end{equation}
where $B(z,r)=\{w: |w-z|<r\}$.
\end{proposition}

A \emph{quasicircle} is defined as the image of a circle under a quasiconformal mapping of the extended complex plane $\overline{\mathbf{C}}\cong S^2$. See \cite{alfors} for comprehensive study of related problems.

\begin{lemma}\label{tomi}
If $D$ is a $k'$ quasicircle and $f$ a $k$-quasiconformal harmonic mapping of the unit disk onto $D$, then the image of every diameter $d=e^{it}[-1,1]$, $t\in[0,\pi)$, under $f$ is
is a chord-arc curve with a constant $C$ depending only on $k$ and $k'$.
\end{lemma}
\begin{proof}
We observe that $f$ extends to a global quasiconformal homeomorphism.

We have to consider the integrals $$\int_{r_0}^1 |\partial_r f(re^{it})|dr,$$ and to show that they are comparable to $\rho=\mathrm{dist}(f(r_0e^{it}),f(e^{it}))$.

By the distortion inequality (see e.g. \cite[p.~224]{kenig2}), for sufficiently large $M$ we have
\begin{equation}
|z_1-z_2|\le M^{-k}|z_2-z_3| \Rightarrow |f(z_1)-f(z_2)|\le 2^{-k}|f(z_2)-f(z_3)|.
\end{equation}
For $r_0$ given define $r_n$, $n\ge 1$, inductively by $r_n\in[0,1]$ such that $$1-r_n=\frac{1}{M}(1-r_{n-1}),$$   and  $$|f(r_{n+1}e^{it})-f(e^{it})|\le \frac{1}{2}|f(r_ne^{it})-f(e^{it})|.$$ Thus  $$|f(r_{n+1}e^{it})-f(e^{it})|\le \frac{1}{2^{n+1}}|f(r_0e^{it})-f(e^{it})|,$$ and for  $r_{n}\le r \le r_{n+1}$ $$|f(re^{it})-f(e^{it})|\le \frac{1}{2^n}|f(r_0e^{it})-f(e^{it})|.$$
So if $r\in [r_n,r_{n+1}]$ we have that
$$\mathrm{dist}(f(re^{it}),\partial D)\le \frac{1}{2^n}|f(r_0e^{it})-f(e^{it})|.$$

Now from Lemma~\ref{leqa} we have  \[\begin{split}\int_{r_0}^1 |\partial_r f(re^{it})|dr&\le \int_{r_0}^1 |\nabla f(re^{it})|dr\le C\int_s^1 \frac{\mathrm{dist}(f(re^{it}),\partial D)}{1-r} dr\\&\le
C\sum_{k=0}^\infty \int_{r_k}^{r_{k+1}} \frac{\mathrm{dist}(f(re^{it}),\partial D)}{1-r} dr\\&\le  C\sum_{k=0}^\infty \int_{r_k}^{r_{k+1}}
 \frac{2^{-k}\rho}{1-r} dr\\&=C\sum_{k=0}^\infty 2^{-k}\log\frac{1-r_{k}}{1-r_{k+1}}{\rho}
\le C\log M\sum_{k=0}^\infty {2^{-k}\rho}\\&=2 C M \mathrm{dist}(f(r_0e^{it}),f(e^{it})).\end{split}\]

Next we have \[\begin{split}\int_{a}^b |\partial_r f(re^{it})|dr&=\int_0^b|\partial_r f(re^{it})|dr-\int_0^a|\partial_r f(re^{it})|dr\\&\le
2CM \mathrm{dist}(f(be^{it}),1)-\mathrm{dist}(f(ae^{it}),f(e^{it}))\\&\le 2CM(\mathrm{dist}(f(be^{it}),f(e^{it}))-\mathrm{dist}(f(ae^{it}),f(e^{it})))\\&\le 2CM\mathrm{dist}(f(be^{it}),f(ae^{it})).\end{split}\]
So $|f(e^{it}[a,b])|\le C |a-b|$ for every $t\in [0,\pi)$ and $a$ and $b$ in $[-1,1]$.
\end{proof}

\section{Proof of main results}

\begin{proof}[Proof of Theorem~\ref{onemaine}]
In the course of a proof, the value of a constant $C$ may change from one
occurrence to the next.

If $f$ is a $K-$quasiconformal harmonic mapping of the unit disk onto a chord-arc Jordan domain $\Omega$, then it exists a constant $K_1=K_1(K)$ and a $K_1-$quasiconformal extension of $f$ onto the whole space. We will denote the extension by $f$ and $K_1$ by $K$.

Let $I\subset \mathbf{T}$ with $|I|<1/2$. By appropriate rotation, we can assume that $I=\{e^{is}: 0\le s\le |I|\}$. Let $R_I=\exp([-|I|,0]\times [0,|I|])$. Let $\zeta_I$ be the mid-point of $I$, let $w_I=\exp(|I|(-1+i)/2)$ and let
$\varrho=\mathrm{dist}(w_I,\partial D)$. By applying the distortion theorem, there is $C>0$ depending only on $K$ such that
\begin{equation}\label{qc}D(f(w_I),\varrho/C)\subset f(R_I)\subset  D(f(w_I),C\varrho).\end{equation}
Then by the chord-arc condition and \eqref{qc} we infer that $$
\sigma(f(I))\cong\mathrm{diam}( D(f(w_I),C\varrho))\cong \varrho.$$ From \eqref{qc} and Lemma~\ref{leqa} we obtain that
the length of $f(e^{-|I|}I)$ is smaller than $C \varrho$.  By using Lemma~\ref{tomi}, we obtain that two other sides of the quadrilateral $\partial f(R_I)$ have length $\le C \varrho$.

Let $g$ be the q.c. harmonic mapping defined by $$g(z)=\frac{\sqrt{(1+K^2)/6}\pi C}{\varrho} f[\exp(z|I|-|I|)].$$  Then $g$ is a $k-$q.c. harmonic mapping of the unit square $\mathbf{K}$ onto the chord-arc Jordan domain $D'$ containing the disk $D(0,C_K)$ and we can apply Corollary~\ref{apo1}. Since $\exp$ is a bi-Lipschitz mapping with an absolute constant  on $[-|I|,0]\times [0,|I|]$, we obtain that for every $\varepsilon>0$ there exists $\delta=\delta(\epsilon)$, such that every $F\subset \partial D$ with $\sigma(g(F))<\delta$ we have $|F|\le \varepsilon \sigma(\partial D').$ Since $|I|=|Q_I|/4$ and $\sigma(f(I))\cong \sigma(f(\partial Q_I))$, it follows that for every $\varepsilon$ there exists $\delta=\delta(\varepsilon)$ such that if $E\subset I$ is measurable, then we have $$\frac{\sigma(f(E))}{\sigma(f(I))}\le \delta\Rightarrow \frac{|E|}{|I|}\le \varepsilon.$$ It follows that the measures $\sigma(E)$ and $|E|$ are $A_\infty$ equivalent or what is the same it is satisfied the  Coifman-Fefferman ($A_\infty$) condition (\cite{coif}, \cite[p.~168]{garnet}) for the weight $\omega(e^{i\varphi})=|\partial_\varphi f(e^{i\varphi})|$, which is equivalent with the Muckenhoupt ($A_p$) condition (\cite{muc}).

\end{proof}
\begin{proof}[Proof of Theorem~\ref{lindel}]
a) Assume that $\gamma\in C^1$. Prove first that the function $\arg(f_t(re^{it})/z)$ is well-defined and smooth on $\mathbf{U}^*$.
It is well-known that every continuous function $P:\Omega\to \mathbf{C}^*$, defined in a simply-connected domain $\Omega$,  has a unique continuous logarithm, $Q:\Omega\to \mathbf{C}$, up to a normalization condition $Q(z_0)=w_0$. This means that there is a mapping $Q$ such that  $e^{Q(z)}=P(z)$ and $e^{Q(z_0)}=P(z_0)$ and we write it as $Q(z)=\log P(z)$.

Let $$H(z)=\frac{\partial_\varphi f(z)}{z}.$$ Then $$H(z)=ia'(z)\left(1-\frac{\overline{zb'(z)}}{za'(z)}\right).$$
So $$\log(H(z))=\log(ia'(z))-\sum_{k=1}^\infty \frac{m(z)^k}{k},$$ where $$m(z)=\frac{\overline{zb'(z)}}{za'(z)}$$ satisfies the condition $$|m(z)|\le \frac{K-1}{K+1}<1,$$ is well defined in $\mathbf{U}^*$. Since $U(z)=\mathrm{Im}(\log(H(z))),$ it follows that $U(z)$ is well-defined smooth function in $\mathbf{U}^*$.

Let us prove now that $\arg[f_t(re^{it})/z]$ has a continuous extension to $\mathbf{T}$.

Without loss of generality we will assume in the proof that $\kappa$ from Theorem~\ref{poqa} is $\le 1$.
Let $z=re^{i\varphi}$. Moreover, assume without loosing of generality that $|\gamma|=2\pi$.

Let $g$ be an arc-length parametrization of $\gamma$.  Since $\gamma\in C^1$, its arc-length parametrization $g$ is in $C^1$ and $g'(s)=e^{i\theta(s)}$, where $\theta$ is continuous in $[0,2\pi]$. Moreover $\theta$ satisfies the condition $\theta(2\pi)-\theta(0)=2\pi$, and therefore it has a natural extension to $\mathbf{R}$:  $\theta(x+2k\pi)=:2 k\pi+\theta(x)$, $k\in\mathbf{Z}$.  Since $f^*$ is a homeomorphism,  there is a homeomorphism $\psi:\mathbf{R}\to \mathbf{R}$ with $\psi(2\pi)-\psi(0)=2\pi$, such that $f^*(e^{it})=g(\psi(t))$. Then we have
\begin{equation}\label{beta}\partial_t f^*(e^{it})=g'(\psi(t))\psi'(t)=e^{i\beta(t)}\psi'(t),\end{equation} for some continuous function $\beta$.

By integration by parts in the Poisson integral formula \eqref{e:POISSON}, in view of Proposition~\ref{propi} we obtain \begin{equation}\label{e:POISSONw}
\frac{\partial \mathcal{P}[f^*](z)}{\partial \varphi}=\int_0^{2\pi}P(r,x-\varphi)\partial_x f^*(e^{ix})dx, \ \ z=re^{i\varphi},\quad 0\le r<1.
\end{equation}

Define  \[\begin{split}V(re^{i\varphi})&:=\arg(A(z)+i B(z))\\&=\arg\left(e^{-i\beta(\varphi)}\partial_t f(z)\right)\\&=\arg\left(e^{-i\beta(\varphi)}\int_0^{2\pi} P(r,t-\varphi)  {df^*(e^{it})}\right)\\&=\arg\left(\int_0^{2\pi} P(r,t-\varphi) e^{i(\beta(t)-\beta(\varphi))} d\psi(t)\right)=\arctan\frac{A}{B},\end{split}\]
where $$A=\int_0^{2\pi} P(r,t) \sin (\beta(t+\varphi)-\beta(\varphi)) \psi'(t+\varphi)dt$$ and
 $$B=\int_0^{2\pi} P(r,t) \cos (\beta(t+\varphi)-\beta(\varphi)) \psi'(t+\varphi)dt.$$
 Prove now that $V$ has a continuous extension to $\mathbf{T}$ with \begin{equation}\label{vv}V(e^{i\varphi})\equiv 0.\end{equation} This statement is equivalent with the main conclusion of our theorem.

 Prove that $\lim_{|z|\to 1}V(z)=0$, where the limit is unrestricted. We will prove that for given $0<\epsilon\le 2$ there is $\delta=\delta(\epsilon)$ such that if $0<1-|z|<\delta$  we have $$\frac{|A(z)|}{|B(z)|}\le \epsilon, \quad B(z)>0.$$

 Since $\beta$ is continuous, there is $\varepsilon=\varepsilon(\epsilon)>0$ such that

 \begin{equation}\label{sinus}|\sin[\beta(\varphi+t)-\beta(\varphi)]|\le \frac{\epsilon}{8\|\psi'\|_{L^1}{ \|1/\psi'\|_{L^{\kappa}}}}, \quad  |e^{it}-1|\le \varepsilon. \end{equation}

 Further,  we have \begin{equation}\label{ac}\int_{|e^{it}-1|>\varepsilon} P(r,t) \psi'(t+\varphi)dt\le \frac{1-r^2}{\varepsilon^2}\|\psi'\|.\end{equation} Thus  \begin{equation}\label{AA}\begin{split}
 |A(z)|&\le \int_0^{2\pi} P(r,t)|\sin[\beta(t+\varphi)-\beta(\varphi)]|\psi'(t+\varphi)dt\\&=
 \int_{|e^{it}-1|>\varepsilon} P(r,t)|\sin[\beta(t+\varphi)-\beta(\varphi)]|\psi'(t+\varphi)dt\\&+\int_{|e^{it}-1|\le \varepsilon} P(r,t)|\sin[\beta(t+\varphi)-\beta(\varphi)]|\psi'(t+\varphi)dt\\&\le
 \frac{1-r^2}{\varepsilon^2}\|\psi'\|+\frac{\epsilon}{8\|\psi'\|\left({ \|1/\psi'\|_{L^{\kappa}}}\right)^{\frac{1+{\kappa}}{{\kappa}}}}\|\psi'\|.
 \end{split}
 \end{equation}
  Since $\beta$ is continuous  there is $\varepsilon>0$ satisfying \eqref{sinus} and
 \begin{equation}\label{cosinus}\cos[\beta(t+\varphi)-\beta(\varphi)]>\frac{1}{2}, \quad  |e^{it}-1|\le \varepsilon.\end{equation}
So we have  \[\begin{split}\int_{|e^{it}-1|\le\varepsilon} P(r,t) &\cos[\beta(t+\varphi)-\beta(\varphi)] \psi'(t+\varphi)dt\\
 &\ge \frac{1}{2}\int_{|e^{it}-1|\le\varepsilon}P(r,t) \psi'(t+\varphi)dt.\end{split}\]

Since $\frac{1}{\psi'(t)}\in L^\kappa[0,2\pi]$, by using H\"older inequality with coefficients  $p=\frac{1+\kappa}{\kappa}$ and $q=\kappa+1$, (c.f. \eqref{pq}) we obtain
\[\begin{split}&\int_{|e^{it}-1|\le \varepsilon}P(r,t) dt\\&\le \left(\int_{|e^{it}-1|\le \varepsilon}P(r,t)\psi'(t+\varphi) dt\right)^{\frac{{\kappa}}{1+{\kappa}}}
  \left(\int_{|e^{it}-1|\le \varepsilon}P(r,t)|\psi'(t+\varphi)|^{-{\kappa}} dt\right)^{\frac{1}{1+{\kappa}}}\\&\le \left(\int_{|e^{it}-1|\le \varepsilon}P(r,t) \psi'(t+\varphi)dt\right)^{\frac{{\kappa}}{1+{\kappa}}}
  \left(\int_0^{2\pi}P(r,t)|\psi'(t+\varphi)|^{-{\kappa}} dt\right)^{\frac{1}{1+{\kappa}}}\\&\le\|1/(\psi')^\kappa\|^{\frac{{1}}{1+{\kappa}}}_{L^{1}}\left(\int_{|e^{it}-1|\le \varepsilon}P(r,t) \psi'(t+\varphi)dt\right)^{\frac{{\kappa}}{1+{\kappa}}}.\end{split}\]
Thus \begin{equation}\label{prima}
\int_{|e^{it}-1|\le \varepsilon}P(r,t) \psi'(t+\varphi)dt\ge\frac{\left(\int_{|e^{it}-1|\le \varepsilon}P(r,t) dt\right)^{\frac{1+{\kappa}}{{\kappa}}}}{ \|1/(\psi')^\kappa\|_{L^{1}}}.
\end{equation}
By \eqref{ac} and \eqref{prima} we obtain
\begin{equation}\label{BB}
B(z)\ge \frac{\left(\int_{|e^{it}-1|\le \varepsilon}P(r,t) dt\right)^{\frac{1+{\kappa}}{{\kappa}}}}{ \|1/(\psi')^\kappa\|_{L^{1}}}-\frac{1-r^2}{\varepsilon^2}\|\psi'\|.
\end{equation}
From \eqref{AA} and \eqref{BB} we obtain
\begin{equation}\label{AB}\begin{split}
\frac{|A(z)|}{|B(z)|}&\le \frac{\frac{1-r^2}{\varepsilon^2}\|\psi'\|+\frac{\epsilon}{8\|\psi'\|\left({ \|1/(\psi')^\kappa\|_{L^{1}}}\right)^{\frac{1+{\kappa}}{{\kappa}}}}\|\psi'\|
}{\frac{\left(\int_{|e^{it}-1|\le \varepsilon}P(r,t) dt\right)^{\frac{1+{\kappa}}{{\kappa}}}}{ \|1/(\psi')^\kappa\|_{L^{1}}}-\frac{1-r^2}{\varepsilon^2}\|\psi'\|}\\&= \frac{\frac{1-r^2}{\varepsilon^2}+\frac{\epsilon}{8\|\left({ \|1/(\psi')^\kappa\|_{L^{1}}}\right)^{\frac{1+{\kappa}}{{\kappa}}}}
}{I-\frac{1-r^2}{\varepsilon^2}},\end{split}
\end{equation}
where $$I=\frac{\left(\int_{|e^{it}-1|\le \varepsilon}P(r,t) dt\right)^{\frac{1+{\kappa}}{{\kappa}}}}{ \|\psi'\|\|1/(\psi')^\kappa\|_{L^{1}}}.$$
 Since $$\int_{|e^{it}-1|\le \varepsilon} P(r,t)dt=\frac{2}{\pi} \arctan\left(\frac{1 + r }{1-r}\frac{\varepsilon}{\sqrt{2-\varepsilon} \sqrt{2+\varepsilon}}\right)$$ it follows that there is $\rho>0$ such that for $r>\rho$, $$\left(\int_{|e^{it}-1|\le \varepsilon} P(r,t)dt\right)^{\frac{1+\kappa}{\kappa}}\ge \frac{1}{2}.$$
 We have that for $r>\rho$  $$I\ge \frac{1}{2\|\psi'\|{ \|1/(\psi')^\kappa\|_{L^{1}}}}.$$
 Hence \begin{equation}\label{AB1}\begin{split}
\frac{|A(z)|}{|B(z)|}&\le  \frac{\frac{\epsilon}{8\|\psi'\|{ \|1/(\psi')^\kappa\|_{L^{1}}}}+\frac{1-r^2}{\varepsilon^2}
}{I-\frac{1-r^2}{\varepsilon^2}}.\end{split}
\end{equation}
  Also we can assume that $\rho$ is such that for $r>\rho$ we have \begin{equation}\label{vare}\frac{1-r^2}{\varepsilon^2}\le \frac{\epsilon}{8\|\psi'\|{ \|1/(\psi')^\kappa\|_{L^{1}}}}.\end{equation}
 From \eqref{AB1} and \eqref{vare} we obtain for $r=|z|>\rho$ \begin{equation}
\frac{|A(z)|}{|B(z)|}\le\frac{\epsilon}{2}+\frac{\epsilon}{2}=\epsilon.\end{equation}

b) The converse part of theorem is elementary and we do not need for the mapping to be quasiconformal.  See the book of Pommerenke \cite[p.~44]{boundary} for its counterpart for conformal mappings. Assume that $\arg(\frac{\partial_tf(re^{it})}{z})$ has a continuous extension to the boundary and prove that $\gamma$ is $C^1$.
Let $$V(z) = \arg \left(\frac{\partial_tf(z)}{z}\right)$$ be continuous in $\overline{\mathbf{U}}\setminus\{0\}$ and let $1/2<r<1$.
 For fixed $z=re^{is}$ define
$$W(re^{it}) =:\frac{f(re^{it}) - f(re^{is})}{t-s}e^{-iV(z)}=\int_s^t\frac{\partial f(re^{i\tau})}{\partial \tau} e^{-iV(z)} \frac{d\tau}{t-s}.$$
$$=\int_s^t\left|\frac{\partial f(re^{i\tau})}{\partial \tau}\right| e^{i(V(re^{i\tau})-V(re^{is}))} \frac{d\tau}{t-s}.$$

Since $V$ is uniformly continuous in $1/2\le |z|\le 1$, it follows that for $\varepsilon>0$ ($\varepsilon<\pi/2$), there is $\delta=\delta(\varepsilon)$ such that if
$1/2\le  |z|, |w|\le 1$ and $|z-w|<\delta$, then $|V(z)-V(w)|<\varepsilon$. Hence if $|t-s|<\delta$ and $r\ge 1/2$ we have
$$|\sin W(re^{it})|\le |\sin\varepsilon|\int_s^t\left|\frac{\partial f(re^{i\tau})}{\partial \tau}\right|  \frac{d\tau}{t-s}$$ and
$$\cos W(re^{it})\ge \cos\varepsilon\int_s^t\left|\frac{\partial f(re^{i\tau})}{\partial \tau}\right|  \frac{d\tau}{t-s}.$$

Therefore, for $|t-s|<\delta$ and $1/2\le r\le 1$ we have $$|\arg W(re^{it})|=|\arctan \frac{\sin W(re^{it})}{\cos W(re^{it})}|\le \varepsilon.$$
This means in particular that if $|t-s|<\delta$ then $$\left|\arg\frac{f(e^{it}) - f(e^{is})}{t-s}-V(e^{is})\right|\le \varepsilon.$$
It follows that $\gamma$ has at $f(e^{is})$ a tangent with direction angle $V(e^{is})$ which varies continuously.

 \end{proof}
 \subsection{Acknowledgement} I would like to thank to professor Carlos Kenig for useful discussions.

\end{document}